\documentclass[11pt]{amsart}
\usepackage{graphicx}
\usepackage{hyperref}
\usepackage{todonotes}
\usepackage{amsmath,amsthm,amssymb}

\usepackage[T1]{fontenc}

\usepackage{listings}
\newtheorem{theorem}{Theorem}

\newtheorem{lemma}[theorem]{Lemma}
\newtheorem{proposition}[theorem]{Proposition}
\theoremstyle{definition}
\newtheorem{definition}[theorem]{Definition}
\theoremstyle{remark}
\newtheorem{remark}[theorem]{Remark}

\begin{document}

\title[Asymptotics for Kendall's renewal function]{Asymptotics for Kendall's renewal function}%
\author[M. Cadena, B. H. Jasiulis-Go{\l}dyn, and E. Omey]{{\bf Meitner Cadena}\\
{\it DECE, Universidad de las Fuerzas Armadas, Sangolqui, Ecuador}
%\and
\\
{\bf Barbara H. Jasiulis-Go{\l}dyn}\\
{\it Institute of Mathematics, University of Wroc{\l}aw, pl. Grunwaldzki 2/4, 50-384 Wroc{\l}aw, Poland}
%\and
\\
{\bf Edward Omey}\\
{\it Faculty of Economics and Business-Campus Brussels, KU Leuven, Warmoesberg 26, 1000 Brussels, Belgium}\\
}%
\email{\newline Meitner Cadena <mncadena2@espe.edu.ec>\newline
Barbara H. Jasiulis-Go{\l}dyn <jasiulis@math.uni.wroc.pl>\newline
Edward Omey <edward.omey@kuleuven.be>
}%

%\thanks{}%
%\date{}%
%\dedicatory{}%
%\commby{}%
% ----------------------------------------------------------------
\begin{abstract}
An elementary renewal theorem and a Blackwell theorem provided by Jasiulis-Go{\l}dyn et al. (2020) in a setting of Kendall convolutions
are proved under weaker hypothesis and extended to the Gamma class.
Convergence rates of the limits concerned in these theorems are analyzed.
\end{abstract}
\maketitle

\section{Introduction}

Recently Jasiulis-Go{\l}dyn, Misiewicz, Naskr{\k e}t \& Omey have formulated a renewal theory for a class of extremal Markov sequences connected with the
Kendall convolution \cite{Jasiulis}.
Such a convolution that is based on the
Williamson transform of a random variable is an example of a generalized convolution introduced by Urbanik \cite{Urbanik}.

Consider a sequence of occurrences of events $T_1$, $T_2$, \ldots, where $T_i$ are independent and identically distributed random variables
with common distribution function $F$.
The sums $S_n=T_1+\cdots+T_n$, with the convention $S_0=0$, then form a sequence called a renewal sequence if $F(0)=0$.
On the other hand, $R(t)=\sum_{n=0}^{\infty}P(S_n\leq t)$ is the renewal function.
For more details on renewal theory, we refer the reader to \cite{Asmussen} and e.g. \cite{Mitov,Rolski}.
These authors studied $R(t)$ when the usual convolutions involved in this function are interchanged by Kendall convolutions. In \cite{Jasiulis} the renewal function was defined by $R_{\star}(t)=\sum_{n=1}^{\infty}P(S_n\leq t)$. Here we will investigate asymptotic behaviour of $R(t) = 1 + R_{\star}(t)$ in the Kendall convolution sense, similarly to the classical renewal theory, which seems to be much more convenient for the new results.
In this setting, among other results, assuming $m(\alpha)=E(T_1^{\alpha})<\infty$, they proved the elementary renewal theorem

\begin{equation}\label{eq01}
\lim_{x\to\infty}\frac{R_{\star}(x)}{x^{\alpha}}=\frac{2}{m(\alpha)}\textrm{,}
\end{equation}

and the Blackwell theorem

\begin{equation}\label{eq02}
\lim_{x\to\infty}\left(\frac{R_{\star}(x+h)}{(x+h)^{\alpha-1}}-\frac{R_{\star}(x)}{x^{\alpha-1}}\right)=\frac{2h}{m(\alpha)}\textrm{,}
\end{equation}

for some $h\in\mathbb{R}$.
Interestingly, these results were obtained when $\overline{F}=1-F$ is regularly varying.
%A function $f$ is regularly varying if for any $t>0$ and for some $\alpha\in\mathbb{R}$, $f(tx)\big/f(x)\to t^{\alpha}$.

In this paper we prove (\ref{eq01}) without considering the assumption on regular variation mentioned above
and prove (\ref{eq02}) by introducing a general condition on $\overline{F}$.
It is proved that such a general condition is satisfied when $\overline{F}$ is regularly varying or belongs to the Gamma class.
We also analyze convergence rates of the limits (\ref{eq01}) and (\ref{eq02}) in the cases where $\overline{F}$ is regularly varying or belongs to the Gamma class.
Our new results on convergence rates contribute to other obtained by a number of scholars in renewal theory, for instance
\cite{Rogozin1973}, \cite{Kalashnikov}, \cite{Rogozin1977} and \cite{Grubel}.

%generated by such a Kendall convolution. This convolution is based on the
%Williamson transform of a random variable $X$. The elementary renewal
%theorem states that
%\begin{equation*}
%\frac{R(x)}{x^{\alpha }}\rightarrow \frac{2}{m(\alpha )}
%\end{equation*}%
%where $\alpha >0$ and $m(\alpha )$ is the $\alpha -$th moment of the basic
%underlying random variable.

In the %first part
next section
of the paper we give an overview of important definitions
and properties that are going to be used.
For instance, Kendall convolution, Williamson transform, renewal function, the class of regularly varying functions and the Gamma class of functions.
In Section \ref{ratesofconvergenceintherenewaltheorem} we study rates of convergence in the
renewal theorem. It turns out that we have to consider two important cases.\
The first case is the case where the tail of the random variable $X$ is regularly
varying. The second case is the case where the tail of $X$ belongs to the
Gamma class of distribution functions. Then, in Section \ref{Blackwell} we study Blackwell type of
results, i.e. we study the asymptotic behaviour of $R(x+y)-R(x)$ and that of
the derivative $R^{\prime }(x)$. Again we study rates of convergence and we
distinguish between the two cases as in the previous section.
Section \ref{ratesofconvergenceBlackwell} presents these last results on rates of convergence.
We close the
paper with some concluding remarks given in the last section.

\section{Notations and definitions}

We write $f(x)\sim g(x)$ by $f(x)\big/g(x)\to1$ as $x\to\infty$.

\subsection{Transforms}

In what follows $X$ is a positive random variable (r.v.) with distribution function (d.f.) $F(x)=P(X\leq x)$ and
$F(0-)=0$.
$F$ is assumed continuous.
The tail of $F$ is given by $\overline{F}(x)=1-F(x)$. Let $\alpha$ denote a positive real number. We define truncated moments as follows

\begin{equation*}
H_{\alpha }(x)=\int_{0}^{x}y^{\alpha }dF(y)\text{,}
\end{equation*}%
and

\begin{equation*}
W_{\alpha }(x)=\int_{0}^{x}y^{\alpha -1}\overline{F}(y)dy\text{.}
\end{equation*}%
Throughout we assume that the $\alpha -$th moment is finite, and we write $%
m(\alpha )=H_{\alpha }(\infty )=\alpha W_{\alpha }(\infty )$. The case of $%
m(\alpha )=\infty $ has been treated in \cite{Jasiulis} and
in \cite{Omey}, \cite{Kevei}. Note that by dominated convergence
we have $\lim_{x\rightarrow \infty }x^{-\alpha }H_{\alpha }(x)=0$.

The Williamson or $G-$transform is given by%
\begin{equation*}
G_{F}(x)=\int_{0}^{x}\left(1-\left(\frac{t}{x}\right)^{\alpha }\right)dF(t)\text{.}
\end{equation*}%
See \cite{Williamson} and e.g. \cite{Jasiulis}. The
probabilistic interpretation of $G_{F}(x)$ is as follows. Let $Z$ denote a
positive r.v. with $P(Z\leq x)=x^{\alpha },0\leq x\leq 1$. In this case we
see that $G_{F}(x)=\int_{0}^{x}P(Z\geq t/x)dF(t)=P(Z\geq X/x)$ or $%
G_{F}(x)=P(X/Z\leq x)$. Hence $G_{F}(x)$ is again a d.f. (with $G_{F}(0-)=0$%
).

For nondecreasing functions $B(x)$ with $B(0-)=0$ we can also define its $G-$%
transform as $G_{B}(x)=\int_{0}^{x}(1-(\frac{t}{x})^{\alpha })dB(t)$.

\subsection{Relationships}

Using partial integration we have $G_{F}(x)=\alpha x^{-\alpha
}\int_{0}^{x}t^{\alpha -1}F(t)dt$, and then also that%
\begin{equation*}
\overline{G}_{F}(x)=\alpha x^{-\alpha }\int_{0}^{x}t^{\alpha -1}\overline{F}%
(t)dt=\alpha x^{-\alpha }W_{\alpha }(x)\text{,}
\end{equation*}%
where $\overline{G}_{F}(x)=1-G_{F}(x)$. Also we have the following inversion
formula, cf. \cite{Jasiulis2} and \cite{Jasiulis}:%
\begin{equation*}
F(x)=G_{F}(x)+\frac{x}{\alpha }G_{F}^{\prime }(x)\text{.}
\end{equation*}%
For nondecreasing functions $B(x)$ with $B(0)=0$, we equally have
\begin{equation*}
B(x)=G_{B}(x)+\frac{x}{\alpha }G_{B}^{\prime }(x)\text{.}
\end{equation*}

The functions introduced above are closely related to each other. Using
partial integration we have%
\begin{equation*}
H_{\alpha }(x)=\alpha \int_{0}^{x}y^{\alpha -1}\overline{F}(y)dy-x^{\alpha }%
\overline{F}(x)=\alpha W_{\alpha }(x)-x^{\alpha }\overline{F}(x)\text{.}
\end{equation*}%
One can show cf \cite{Omey} that we can find back $\overline{F}%
(x)$ from $H_{\alpha }(x)$ and we have:%
\begin{equation*}
\overline{F}(x)=\alpha \int_{x}^{\infty }z^{-\alpha -1}H_{\alpha
}(z)dz-x^{-\alpha }H_{\alpha }(x)\text{.}
\end{equation*}%
Also we have $\alpha W_{\alpha }(x)=H_{\alpha }(x)+x^{\alpha }\overline{F}%
(x)$. In terms of $H_{\alpha }(x)$, we have $G_{F}(x)=F(x)-x^{-\alpha
}H_{\alpha }(x)$ and
\begin{equation*}
\overline{G}_{F}(x)=\overline{F}(x)+x^{-\alpha }H_{\alpha }(x)\text{.}
\end{equation*}

Since $m(\alpha )<\infty $, we use the following notations and identities:%
\begin{eqnarray*}
\overline{W}_{\alpha }(x) &=&W_{\alpha }(\infty )-W_{\alpha
}(x)=\int_{x}^{\infty }y^{\alpha -1}\overline{F}(y)dy\text{,} \\
\overline{H}_{\alpha }(x) &=&H_{\alpha }(\infty )-H_{\alpha }(x)=\alpha
\overline{W}_{\alpha }(x)+x^{\alpha }\overline{F}(x)\text{,} \\
\alpha \overline{W}_{\alpha }(x) &=&m(\alpha )-x^{\alpha }\overline{G}_{F}(x)
\end{eqnarray*}

We summarize some of the formulas that we obtained in the following
propositions. Cf. \cite{Omey}.

\begin{proposition}
%(1)
We have:

\begin{itemize}
	\item[(i)] $H_{\alpha }(x)=\alpha \int_{0}^{x}t^{\alpha -1}\overline{F}%
(t)dt-x^{\alpha }\overline{F}(x)$.

	\item[(ii)] $\overline{F}(x)=\alpha \int_{x}^{\infty }z^{-\alpha -1}H_{\alpha
}(z)dz-x^{-\alpha }H_{\alpha }(x)$.

	\item[(iii)] $\overline{F}(x)=\overline{G}_{F}(x)-x^{-\alpha }H_{\alpha }(x)$.

	\item[(iv)] $\overline{G}_{F}(x)=\alpha x^{-\alpha }W_{\alpha }(x)$.

\end{itemize}
\end{proposition}

For the 'tails' we have the following relationships.

\begin{proposition}
%(2)
If $m(\alpha )<\infty $, we have:

\begin{itemize}
	\item[(i)] $\overline{W}_{\alpha }(x)=\int_{x}^{\infty }y^{\alpha -1}\overline{F}%
(y)dy$.

\item[(ii)] $\overline{H}_{\alpha }(x)=\alpha \int_{x}^{\infty }t^{\alpha -1}%
\overline{F}(t)dt+x^{\alpha }\overline{F}(x)$.

\item[(iii)] $x^{-\alpha }m(\alpha )-\overline{G}_{F}(x)=\alpha x^{-\alpha }%
\overline{W}_{\alpha }(x)$.
\end{itemize}
\end{proposition}

\subsection{Regular variation and the Gamma class }

For relevant background information about regular variation and the Gamma
class, we refer to \cite{Bingham} and \cite{Geluk}. Here we
recall the main definitions.

\begin{definition}
A positive and measurable function $f(x)$ is regularly
varying with index $\alpha\in\mathbb{R} $ if for all $t>0$ we have%
\begin{equation*}
\lim_{x\rightarrow \infty }\frac{f(tx)}{f(x)}=t^{\alpha }\text{.}
\end{equation*}
\end{definition}

\begin{definition}
A positive and measurable function $f(x)$ is in the
class $\Gamma (g)$ with auxiliary function $g(x)$ if for all real numbers $x$
we have%
\begin{equation*}
\lim_{x\rightarrow \infty }\frac{f(x+tg(x))}{f(x)}=e^{-t}\text{.}
\end{equation*}

The auxiliary function is $g(x)$ self-neglecting (SN), i.e. $\lim_{x\rightarrow \infty
}g(x+tg(x))\big/g(x)=1$, and satisfies $\lim_{x\rightarrow \infty }g(x)\big/x=0$.
\end{definition}

\subsection{The Kendall convolution}

Starting from two d.f.'s $F_{1}(x)=P(X_{1}\leq x)$ and $F_{2}(x)=P(X_{2}\leq
x)$ with $F_{1}(0-)=F_{2}(0-)=0$, we can find their $G-$transforms $%
G_{1}(x)=G_{F_{1}}(x)$ and $G_{2}(x)=G_{F_{2}}(x)$. Now we consider the
product $A(x)=G_{1}(x)G_{2}(x)$. One can prove that there is a d.f. $%
F_{3}(x)=P(X_{3}\leq x)$ so that $A(x)$ is its $G-$transform
when $F_3$ is the Kendall convolution of $F_1$ and $F_2$, cf.
\cite{Jasiulis2} and \cite{Jasiulis}.
Details on the Kendall convolution can be found in \cite{Kendall} and e.g. \cite{Kucharczak} and \cite{Jasiulis3}.
To stress the dependence of $F_{3}(x)$ on $%
F_{1}(x)$ and $F_{2}(x)$, we use the notation $F_{3}(x)=F_{1}\boxtimes
F_{2}(x)=P(X_{1}\boxtimes X_{2}\leq x)$. %and we call it the Kendall
%convolution of $F_{1}$ and $F_{2}$.
Alternatively we write $X_{3}\overset{d}{%
=}X_{1}\boxtimes X_{2}$.

If $F_{1}(x)=F_{2}(x)=F(x)$, we use the notation $F\boxtimes
F(x)=F^{\boxtimes 2}(x)$ and $S_{\boxtimes 2}=X_{1}\boxtimes X_{2}$. In this way we can
construct the $n-$fold convolution $F^{\boxtimes n}(x)=P(S_{\boxtimes n}\leq
x)$ which has $G-$transform given by $G_{F}^{n}(x)$.

\subsection{Kendall renewal function}

Using the Kendall convolution, the Kendall renewal function is given by
\begin{equation}\label{Kendallrenewalfunction}
R(x)=\sum_{n=0}^{\infty }F^{\boxplus n}(x)\text{.}
\end{equation}%
where $F^{\boxtimes 0}(x)=\delta _{0}(x)$. Its $G-$transform is very simple:
\begin{equation*}
G_{R}(x)=\sum_{n=0}^{\infty }G_{F^{\boxplus n}}(x)=\sum_{n=0}^{\infty
}G_{F}^{n}(x)=\frac{1}{\overline{G}_{F}(x)}\text{.}
\end{equation*}%
The inversion formula then gives%
\begin{equation}
R(x)=\frac{2}{\overline{G}_{F}(x)}-\frac{\overline{F}(x)}{\overline{G}%
_{F}^{2}(x)}\textrm{.}  \label{1}
\end{equation}%
To see this, we make some calculations:%
\begin{eqnarray*}
R(x) &=&G_{R}(x)+\frac{x}{\alpha }G_{R}^{\prime }(x) \\
&=&\frac{1}{\overline{G}_{F}(x)}+\frac{1}{\overline{G}_{F}^{2}(x)}\frac{x}{%
\alpha }G_{F}^{\prime }(x) \\
&=&\frac{1}{\overline{G}_{F}(x)}+\frac{1}{\overline{G}_{F}^{2}(x)}(\overline{%
G}_{F}(x)-\overline{F}(x)) \\
&=&\frac{2}{\overline{G}_{F}(x)}-\frac{\overline{F}(x)}{\overline{G}%
_{F}^{2}(x)}\textrm{.}
\end{eqnarray*}

Using (\ref{1}) we have%
\begin{equation*}
\frac{R(x)}{x^{\alpha }}=\frac{2}{x^{\alpha }\overline{G}_{F}(x)}-\frac{%
x^{\alpha }\overline{F}(x)}{x^{2\alpha }\overline{G}_{F}^{2}(x)}\text{.}
\end{equation*}%
Since $x^{\alpha }\overline{G}_{F}(x)\rightarrow m(\alpha )$ and $x^{\alpha }%
\overline{F}(x)\rightarrow 0$ as $x\to\infty$, we obtain that
\begin{equation}
\lim_{x\to\infty}\frac{R(x)}{x^{\alpha }}= \frac{2}{m(\alpha )}\text{,}  \label{2}
\end{equation}%
cf. \cite{Jasiulis}.
In the next sections we study the rate of
convergence in (\ref{2}). We summarize our findings in the next proposition.

\begin{proposition}\label{prop2}
%(3)
Let $R(x)$ be the Kendall renewal function defined by (\ref{Kendallrenewalfunction}).
We have
\begin{equation*}
R(x)=\frac{2}{\overline{G}_{F}(x)}-\frac{\overline{F}(x)}{\overline{G}%
_{F}^{2}(x)}\text{,}
\end{equation*}%
and
\begin{equation*}
\lim_{x\to\infty}\frac{R(x)}{x^{\alpha }}= \frac{2}{m(\alpha )}\text{.}
\end{equation*}
\end{proposition}

\subsection{Examples}

\begin{enumerate}
	\item %1)
	Let $F_{X}(x)=\delta _{1}(x)$. We have $m(\alpha )=1$ and \\
	$%
G_{F}(x)=\int_{0}^{x}(1-(\frac{t}{x})^{\alpha })d\delta _{1}(t)$. For $x\geq
1$ we find that $G_{F}(x)=1-x^{-\alpha }$ and $\overline{G}%
_{F}(x)=x^{-\alpha }$. Cleary for $x\geq 1$ we have the following simple
formula:
\begin{equation*}
R(x)=\frac{2}{\overline{G}_{F}(x)}-\frac{\overline{F}(x)}{\overline{G}%
_{F}^{2}(x)}=2x^{\alpha }\text{.}
\end{equation*}

\item
%2)
Let $F(x)=1-x^{-\beta },x\geq 1$. We take $\beta >\alpha $ so that $%
m(\alpha )<\infty $. We have $H_{\alpha }(x)=\int_{0}^{x}y^{\alpha }dF(y)$.
For $x\geq 1$, we find
\begin{equation*}
H_{\alpha }(x)=\beta \int_{1}^{x}y^{\alpha -\beta -1}dy=\frac{\beta }{\beta
-\alpha }(1-x^{\alpha -\beta })\text{.}
\end{equation*}%
Using $G_{F}(x)=F(x)-x^{-\alpha }H_{\alpha }(x)$, it follows that for $x\geq
1$:%
\begin{equation*}
G_{F}(x)=1+\frac{\alpha }{\beta -\alpha }x^{-\beta }-\frac{\beta }{\beta
-\alpha }x^{-\alpha }\text{,}
\end{equation*}%
and%
\begin{equation*}
\overline{G}_{F}(x)=\frac{\beta }{\beta -\alpha }x^{-\alpha }-\frac{\alpha }{%
\beta -\alpha }x^{-\beta }\textrm{.}
\end{equation*}%
Note that $x^{\alpha }\overline{G}_{F}(x)\rightarrow m(\alpha )=\beta
\big/(\beta -\alpha )$ as $x\to\infty$, and we can write $\overline{G}_{F}(x)=m(\alpha
)x^{-\alpha }+(1-m(\alpha ))x^{-\beta }$. Also note that $\overline{F}(x)\big/%
\overline{G}_{F}(x)\rightarrow 0$ as $x\to\infty$.

Using (\ref{1}) we have for $x\geq 1$,%
\begin{eqnarray*}
R(x) &=&\frac{2}{m(\alpha )x^{-\alpha }+(1-m(\alpha ))x^{-\beta }} \\
&&-\frac{x^{-\beta }}{(m(\alpha )x^{-\alpha }+(1-m(\alpha ))x^{-\beta })^{2}}
\\
&=&\frac{2x^{\alpha }}{m(\alpha )+(1-m(\alpha ))x^{\alpha -\beta }} \\
&&-\frac{x^{\alpha }x^{\alpha -\beta }}{(m(\alpha )+(1-m(\alpha ))x^{\alpha
-\beta })^{2}}\textrm{.}
\end{eqnarray*}

It follows that $x^{-\alpha }R(x)\rightarrow 2\big/m(\alpha )$ as $x\to\infty$.

As to the rate of convergence, we study the difference:

\begin{eqnarray*}
R(x)-\frac{2x^{\alpha }}{m(\alpha )} &=&\frac{2x^{\alpha }}{m(\alpha
)+(1-m(\alpha ))x^{\alpha -\beta }}-\frac{2x^{\alpha }}{m(\alpha )} \\
&&-\frac{x^{\alpha }x^{\alpha -\beta }}{(m(\alpha )+(1-m(\alpha ))x^{\alpha
-\beta })^{2}} \\
&=&\frac{-2x^{\alpha }(1-m(\alpha ))x^{\alpha -\beta }}{m(\alpha )(m(\alpha
)+(1-m(\alpha ))x^{\alpha -\beta })} \\
&&-\frac{x^{\alpha }x^{\alpha -\beta }}{(m(\alpha )+(1-m(\alpha ))x^{\alpha
-\beta })^{2}}\textrm{,}
\end{eqnarray*}%
and then, as $x\to\infty$,%
\begin{eqnarray*}
x^{\beta -2\alpha }\left(R(x)-\frac{2x^{\alpha }}{m(\alpha )}\right) &\rightarrow &%
\frac{-2(1-m(\alpha ))}{m^{2}(\alpha )}-\frac{1}{m^{2}(\alpha )} \\
&=&\frac{-3+2m(\alpha )}{m^{2}(\alpha )}\text{.}
\end{eqnarray*}

%\bigskip

\item
%3)
Let $F_{X}(x)=1-e^{-x},x\geq 0$ and $\alpha =1$. We have $m(1)=1$ and
straightforward calculations show that
\begin{equation*}
G_{F}(x)=\int_{0}^{x}\left(1-\frac{t}{x}\right)e^{-t}dt=1-\frac{F(x)}{x}\text{.}
\end{equation*}%
Hence $\overline{G}_{F}(x)=F(x)\big/x$. For the renewal function, we find
\begin{eqnarray*}
R(x) &=&\frac{2}{\overline{G}_{F}(x)}-\frac{\overline{F}(x)}{\overline{G}%
_{F}^{2}(x)}=\frac{2x}{F(x)}-\frac{x^{2}\overline{F}(x)}{F^{2}(x)} \\
&=&\frac{x}{F(x)}\left(2-\frac{x\overline{F}(x)}{F(x)}\right)\text{.}
\end{eqnarray*}%
It follows that $R(x)\big/x\rightarrow 2$ as $x\to\infty$.

To find the rate of convergence, we
proceed as follows. We have%
\begin{eqnarray*}
R(x)-2x &=&2x\frac{\overline{F}(x)}{F(x)}-\frac{x^{2}\overline{F}(x)}{%
F^{2}(x)} \\
&=&\frac{x^{2}\overline{F}(x)}{F(x)}\left(\frac{2}{x}-\frac{1}{F(x)}\right)\sim
-x^{2}e^{-x}\text{.}
\end{eqnarray*}%
We find that $R(x)\big/x-2\sim -x\overline{F}(x)$. We also have, as $x\to\infty$,%
\begin{eqnarray*}
x\left(\frac{R(x)-2x}{x^{2}\overline{F}(x)}+1\right) &=&\frac{x}{F(x)}\left(\frac{2}{x}-%
\frac{1}{F(x)}\right)+1 \\
&=&\frac{2}{F(x)}-\frac{x\overline{F}(x)(1+F(x))}{F^{2}(x)} \\
&\rightarrow &2\textrm{.}
\end{eqnarray*}

\item
%4) 
Let us consider distribution with the lack of memory property (for details see \cite{Jasiulis}) in the Kendall convolution algebra with $F_X(x)=x^{\alpha}$, $0<x\leq 1$. Then   $m(\alpha ) = \frac{1}{2}$,  
$$
H_{\alpha }(x) = \int_{0}^{x}y^{\alpha }dF(y) =  \frac{x^{\alpha}}{2} \pmb{1}_{(0,1]}(x) + \frac{1}{2} \pmb{1}_{[1,\infty)}(x)
$$  
and
$$
G_F(x) = \frac{x^{\alpha}}{2} \pmb{1}_{(0,1]}(x) +\left(1 - \frac{1}{2 x^{\alpha}} \right) \pmb{1}_{[1,\infty)}(x).
$$
Consequently
$$
x^{-\alpha} R(x) = 4
$$
and the rate of convergence is equal 0.

\item
%5)
Let $ F_X(x) = \left( 1 + x^{-\alpha} \right) e^{- x^{-\alpha}}, x\geq 0$,
where $\alpha > 0$. Notice that it is the limit distribution (see \cite{ArenJasOmey}) in the Kendall convolution algebra corresponding to normal distribution in the classical case. Then the Williamson transform and truncated $\alpha$-moment for $X$ are given by
\[
G_F(x) = H(x)_{\alpha} = \exp\{- x^{-\alpha}\} \pmb{1}_{(0,\infty)}(x).
\]
Hence, $m(\alpha)=1$ and $x^{-\alpha }R(x)\rightarrow 2$ as $x\to\infty$.
Notice that we have
\begin{eqnarray*}
\overline{F}(x) &=&1+(1+x^{-\alpha })(\overline{G}_F(x)-1) \\
&=&(1+x^{-\alpha })\overline{G}_F(x)-x^{-\alpha }\text{.}
\end{eqnarray*}%
Notice that $\overline{G}_F(x)=1-\exp\{ -x^{-\alpha }\}\sim x^{-\alpha }$. Using $%
1-\exp\{-z\} = z-z^{2}/2(1+o(1))$ as $z\rightarrow 0$, we have
\[
\overline{G}_F(x)=x^{-\alpha }-\frac{1}{2}x^{-2\alpha }(1+o(1))\text{,}
\]
and
\[
1-x^{\alpha }\overline{G}_F(x)=\frac{1}{2}x^{-\alpha }(1+o(1))
\]
Returning to $R(x)$, we have 
\begin{eqnarray*}
R(x) &=&\frac{2}{\overline{G}_F(x)}-\frac{(1+x^{-\alpha })\overline{G}
(x)-x^{-\alpha }}{\overline{G}^{2}(x)} \\
&=&\frac{1}{\overline{G}_F(x)}-\frac{1}{x^{\alpha }\overline{G}_F(x)}+\frac{1}{
x^{\alpha }\overline{G}^{2}(x)}
\end{eqnarray*}

Hence 
\begin{eqnarray*}
x^{-\alpha }R(x)-2 &=& \left(\frac{1}{x^{\alpha }\overline{G}_F(x)}-1\right) + \left(\frac{1}{
x^{2\alpha }\overline{G}^{2}(x)}-1 \right)-\frac{1}{x^{2\alpha }\overline{G}_F(x)} \\
&=& \left(\frac{1}{x^{\alpha }\overline{G}_F(x)}-1 \right) \left(\frac{1}{x^{\alpha }\overline{G}
(x)}+2 \right)-\frac{1}{x^{2\alpha }\overline{G}_F(x)} \\
&=&\frac{1}{x^{\alpha }\overline{G}_F(x)}\left(1-x^{\alpha }\overline{G}_F(x)\right) \left(2+
\frac{1}{x^{\alpha }\overline{G}_F(x)}\right)-\frac{1}{x^{2\alpha }\overline{G}_F(x)}
\end{eqnarray*}
and
\begin{eqnarray*}
x^{\alpha }(x^{-\alpha }R(x)-2) &=&\frac{1}{x^{\alpha }\overline{G}_F(x)}
x^{\alpha }(1-x^{\alpha }\overline{G}_F(x))\left(2+\frac{1}{x^{\alpha }\overline{G}
(x)}\right)-\frac{1}{x^{\alpha }\overline{G}_F(x)} \\
&\rightarrow &\frac{1}{2}\times 3-1=\frac{1}{2}
\end{eqnarray*}

\end{enumerate}
%\pagebreak

\section{Rate of convergence in the renewal theorem}
\label{ratesofconvergenceintherenewaltheorem}

To study the rate of convergence in the renewal theorem (\ref{2}), we start from
(\ref{1}). We have $x^{-\alpha }m(\alpha )-\overline{G}_{F}(x)=\alpha x^{-\alpha }%
\overline{W}_{\alpha }(x)$, cf.\ Proposition \ref{prop2}. Now we write
\begin{eqnarray}
&&R(x)-\frac{2}{x^{-\alpha }m(\alpha )}+\frac{\overline{F}(x)}{(x^{-\alpha
}m(\alpha ))^{2}}  \notag \\
&=&2\left(\frac{1}{\overline{G}_{F}(x)}-\frac{1}{x^{-\alpha }m(\alpha )}\right)-%
\overline{F}(x)\left(\frac{1}{\overline{G}_{F}^{2}(x)}-\frac{1}{(x^{-\alpha
}m(\alpha ))^{2}}\right)  \notag \\
&=&2I-\overline{F}(x)\,II\textrm{.}  \label{3}
\end{eqnarray}

We consider the two terms in (\ref{3}) separately. For the first term we have
\begin{equation*}
I=\frac{\alpha \overline{W}_{\alpha }(x)}{m(\alpha )}\frac{1}{\overline{G}%
_{F}(x)}\textrm{.}
\end{equation*}

Using the same formula again, we also find that%
\begin{equation*}
I-\frac{\alpha x^{\alpha }\overline{W}_{\alpha }(x)}{m^{2}(\alpha )}=\frac{%
\alpha \overline{W}_{\alpha }(x)}{m(\alpha )}\times I=\frac{\alpha ^{2}%
\overline{W}_{\alpha }^{2}(x)}{m^{2}(\alpha )\overline{G}_{F}(x)}\text{.}
\end{equation*}

Since $x^{\alpha }\overline{G}_{F}(x)\rightarrow m(\alpha )$ as $x\to\infty$, it follows that%
\begin{equation*}
I\sim \frac{\alpha x^{\alpha }\overline{W}_{\alpha }(x)}{m^{2}(\alpha )}%
\text{,}
\end{equation*}%
and also that
\begin{equation}
I-\frac{\alpha x^{\alpha }\overline{W}_{\alpha }(x)}{m^{2}(\alpha )}\sim
\frac{\alpha ^{2}x^{\alpha }\overline{W}_{\alpha }^{2}(x)}{m^{3}(\alpha )}\textrm{.}
\label{4}
\end{equation}

Now we consider the second part. We have
\begin{eqnarray*}
II &=&I\times \left(\frac{1}{\overline{G}_{F}(x)}+\frac{1}{x^{-\alpha }m(\alpha )}%
\right) \\
&=&\frac{\alpha \overline{W}_{\alpha }(x)}{m(\alpha )}\frac{x^{2\alpha }}{%
x^{\alpha }\overline{G}_{F}(x)}\left(\frac{1}{x^{\alpha }\overline{G}_{F}(x)}+%
\frac{1}{m(\alpha )}\right)\text{.}
\end{eqnarray*}%
\bigskip

Since $x^{\alpha }\overline{G}_{F}(x)\rightarrow m(\alpha )$ as $x\to\infty$, it follows that%
\begin{equation}
II\sim \frac{2\alpha }{m^{3}(\alpha )}x^{2\alpha }\overline{W}_{\alpha }(x)\textrm{.}
\label{5}
\end{equation}

Combining (\ref{3}),(\ref{4}) and (\ref{5}), we find%
\begin{eqnarray}
&&R(x)-\frac{2x^{\alpha }}{m(\alpha )}-\frac{2\alpha x^{\alpha }\overline{W}%
_{\alpha }(x)}{m^{2}(\alpha )}+\frac{x^{2\alpha }\overline{F}(x)}{%
m^{2}(\alpha )}  \notag \\
&=&2\left(I-\frac{\alpha x^{\alpha }\overline{W}_{\alpha }(x)}{m^{2}(\alpha )}\right)-%
\overline{F}(x)\times II  \notag \\
&=&(1+o(1))\frac{2\alpha ^{2}}{m^{3}(\alpha )}x^{\alpha }\overline{W}%
_{\alpha }^{2}(x)-(1+o(1))\frac{2\alpha }{m^{3}(\alpha )}x^{2\alpha }%
\overline{W}_{\alpha }(x)\overline{F}(x)\textrm{.}  \label{6}
\end{eqnarray}

Now we analyse (\ref{6}) further and we consider two important cases.

\subsection{The case of regular variation}

First suppose that $\overline{F}(x)\in RV_{-\beta }$, $\beta >\alpha $. In
this case, applying Karamata's theorem gives, see \cite{Karamata} and e.g. \cite{deHaan},%
\begin{equation*}
\overline{W}_{\alpha }(x)=\int_{x}^{\infty }y^{\alpha -1}\overline{F}%
(y)dy\sim \frac{1}{\beta -\alpha }x^{\alpha }\overline{F}(x)\textrm{.}
\end{equation*}%
Using (\ref{6}), we have%
\begin{eqnarray*}
&&\frac{1}{x^{3\alpha }\overline{F}^{2}(x)}\left(R(x)-\frac{2x^{\alpha }}{%
m(\alpha )}-\frac{2\alpha x^{\alpha }\overline{W}_{\alpha }(x)}{m^{2}(\alpha
)}+\frac{x^{2\alpha }\overline{F}(x)}{m^{2}(\alpha )}\right) \\
&\rightarrow &\frac{2\alpha ^{2}}{(\beta -\alpha )^{2}m^{3}(\alpha )}-\frac{%
2\alpha }{(\beta -\alpha )m^{3}(\alpha )} \\
&=&\frac{2\alpha (2\alpha -\beta )}{m^{3}(\alpha )(\beta -\alpha )^{2}}\textrm{.}
\end{eqnarray*}

Note that the formula above, implies that, as $x\to\infty$,
\begin{equation*}
R(x)-\frac{2}{m(\alpha )}x^{\alpha }-\frac{2\alpha x^{\alpha }\overline{W}%
_{\alpha }(x)}{m^{2}(\alpha )}+\frac{1}{m^{2}(\alpha )}x^{2\alpha }\overline{%
F}(x)=o(1)x^{2\alpha }\overline{F}(x)\textrm{.}
\end{equation*}

Using $\overline{W}_{\alpha }(x)\sim \frac{1}{\beta -\alpha }x^{\alpha }%
\overline{F}(x)$, we then find that, as $x\to\infty$,
\begin{eqnarray*}
\frac{1}{x^{2\alpha }\overline{F}(x)}\left(R(x)-\frac{2x^{\alpha }}{m(\alpha )}\right)
&\rightarrow &\frac{2\alpha }{(\beta -\alpha )m^{2}(\alpha )}-\frac{1}{%
m^{2}(\alpha )} \\
&=&\frac{3\alpha -\beta }{(\beta -\alpha )m^{2}(\alpha )}\textrm{.}
\end{eqnarray*}

We see that $x^{-\alpha }R(x)\rightarrow 2\big/m(\alpha )$ as $x\to\infty$ with a rate of
convergence determined by $x^{\alpha }\overline{F}(x)$. We summarize

\begin{theorem}
%(4)
Let $R(x)$ be the Kendall renewal function defined by (\ref{Kendallrenewalfunction}).
If $\overline{F}(x)\in RV_{-\beta },\beta >\alpha $, then
\begin{equation*}
\lim_{x\to\infty}\frac{1}{x^{2\alpha }\overline{F}(x)}\left(R(x)-\frac{2}{m(\alpha )}x^{\alpha
}\right)= \frac{3\alpha -\beta }{(\beta -\alpha )m^{2}(\alpha )}\text{,}
\end{equation*}%
and%
\begin{equation*}
\lim_{x\to\infty}\frac{1}{x^{3\alpha }\overline{F}^{2}(x)}\left(R(x)-\frac{2x^{\alpha }}{%
m(\alpha )}-\frac{2\alpha x^{\alpha }\overline{W}_{\alpha }(x)}{m^{2}(\alpha
)}+\frac{x^{2\alpha }\overline{F}(x)}{m^{2}(\alpha )}\right)
=\frac{2\alpha (2\alpha -\beta )}{m^{3}(\alpha )(\beta -\alpha
)^{2}}\text{.}
\end{equation*}%
%\begin{eqnarray*}
%&&\frac{1}{x^{3\alpha }\overline{F}^{2}(x)}(R(x)-\frac{2x^{\alpha }}{%
%m(\alpha )}-\frac{2\alpha x^{\alpha }\overline{W}_{\alpha }(x)}{m^{2}(\alpha
%)}+\frac{x^{2\alpha }\overline{F}(x)}{m^{2}(\alpha )}) \\
%&\rightarrow &\frac{2\alpha (2\alpha -\beta )}{m^{3}(\alpha )(\beta -\alpha
%)^{2}}\text{.}
%\end{eqnarray*}
\end{theorem}

\bigskip

\subsection{The case of the class $\Gamma $}

In the case where $\overline{F}(x)\in \Gamma (g)$, we have $\overline{F}%
(x+yg(x))\big/\overline{F}(x)\rightarrow e^{-y}$ as $x\to\infty$, where $g(x)\in SN$ is an
auxiliary function satisfying $g(x)\big/x\rightarrow 0$ as $x\to\infty$. Clearly also $x^{\alpha
-1}\overline{F}(x)\in \Gamma (g)$, cf. \cite{Bingham} (Chapter 3.10)
or \cite{Geluk} (Chapter I) and we have the following property:

\begin{equation}
\overline{W}_{\alpha }(x)=\int_{x}^{\infty }y^{\alpha -1}\overline{F}%
(y)dy\sim \frac{g(x)}{x}x^{\alpha }\overline{F}(x)\text{.}  \label{7}
\end{equation}%
Using (\ref{6}) and (\ref{7}) leads to%
\begin{eqnarray*}
&&R(x)-\frac{2x^{\alpha }}{m(\alpha )}-\frac{2\alpha x^{\alpha }\overline{W}%
_{\alpha }(x)}{m^{2}(\alpha )}+\frac{x^{2\alpha }\overline{F}(x)}{%
m^{2}(\alpha )} \\
&=&(1+o(1))\frac{g^{2}(x)2\alpha ^{2}x^{3\alpha }\overline{F}^{2}(x)}{%
x^{2}m^{3}(\alpha )}-(1+o(1))\frac{g(x)2\alpha }{xm^{3}(\alpha )}x^{3\alpha }%
\overline{F}^{2}(x) \\
&\sim &-\frac{g(x)}{x}\frac{2\alpha }{m^{3}(\alpha )}x^{3\alpha }\overline{F}%
^{2}(x)\textrm{.}
\end{eqnarray*}%
The last line follows because the first term is dominated by the second term.

\bigskip

Since $x^{\alpha }\overline{F}(x)\rightarrow 0$ as $x\to\infty$, it follows that
\begin{equation*}
R(x)-\frac{2x^{\alpha }}{m(\alpha )}+\frac{x^{2\alpha }\overline{F}(x)}{%
m^{2}(\alpha )}=\frac{2\alpha x^{\alpha }\overline{W}_{\alpha }(x)}{%
m^{2}(\alpha )}+o(1)\frac{g(x)}{x}x^{2\alpha }\overline{F}(x)\textrm{.}
\end{equation*}%
Using $\overline{W}_{\alpha }(x)\sim g(x)x^{\alpha }\overline{F}(x)\big/x$, we
find that, as $x\to\infty$,%
\begin{equation*}
\frac{x}{g(x)x^{2\alpha }\overline{F}(x)}\left(R(x)-\frac{2x^{\alpha }}{m(\alpha )%
}+\frac{x^{2\alpha }\overline{F}(x)}{m^{2}(\alpha )}\right)\rightarrow \frac{%
2\alpha }{m^{2}(\alpha )}\textrm{.}
\end{equation*}

Among others we see that, as $x\to\infty$,%
\begin{equation*}
\frac{1}{x^{\alpha }\overline{F}(x)}\left(x^{-\alpha }R(x)-\frac{2}{m(\alpha )}\right)+%
\frac{1}{m^{2}(\alpha )}\sim \frac{g(x)}{x}\frac{2\alpha }{m^{2}(\alpha )}\textrm{.}
\end{equation*}

and also that, as $x\to\infty$,%
\begin{equation*}
\frac{1}{x^{\alpha }\overline{F}(x)}\left(\frac{2}{m(\alpha )}-x^{-\alpha
}R(x)\right)\rightarrow \frac{1}{m^{2}(\alpha )}\text{,}
\end{equation*}%
and that also here $x^{-\alpha }R(x)\rightarrow 2\big/m(\alpha )$ as $x\to\infty$ at a rate
determined by $x^{\alpha }\overline{F}(x)$. We summarize

\begin{theorem}
%(5)
Let $R(x)$ be the Kendall renewal function defined by (\ref{Kendallrenewalfunction}).
If $\overline{F}(x)\in \Gamma (g)$, then%
\begin{equation*}
\lim_{x\to\infty}\frac{1}{x^{\alpha }\overline{F}(x)}\left(\frac{2}{m(\alpha )}-x^{-\alpha
}R(x)\right)= \frac{1}{m^{2}(\alpha )}\text{,}
\end{equation*}%
and%
\begin{equation*}
\frac{1}{x^{\alpha }\overline{F}(x)}\left(x^{-\alpha }R(x)-\frac{2}{m(\alpha )}\right)+%
\frac{1}{m^{2}(\alpha )}\sim \frac{g(x)}{x}\frac{2\alpha }{m^{2}(\alpha )}%
\text{.}
\end{equation*}
\end{theorem}

%\pagebreak \bigskip

\section{Renewal theorems of Blackwell type}
\label{Blackwell}

The Blackwell renewal theorem studies the asymptotic behaviour of difference
$R(x+y)-R(x)$. In our case, we take $y>0$, and using (\ref{1}), we find that

\begin{eqnarray*}
R(x+y)-R(x) &=&2\left(\frac{1}{\overline{G}_{F}(x+y)}-\frac{1}{\overline{G}_{F}(x)%
}\right)-\left(\frac{\overline{F}(x+y)}{\overline{G}_{F}^{2}(x+y)}-\frac{\overline{F}(x)%
}{\overline{G}_{F}^{2}(x)}\right) \\
&=&2I-II\textrm{.}
\end{eqnarray*}

We consider the two terms separately.

%1)
On the first term we have
\begin{equation*}
I=\frac{1}{\overline{G}_{F}(x+y)}-\frac{1}{\overline{G}_{F}(x)}=\frac{%
G_{F}(x+y)-G_{F}(x)}{\overline{G}_{F}(x)\overline{G}_{F}(x+y)}\textrm{.}
\end{equation*}%
Using the mean value theorem, this gives%
\begin{equation*}
\frac{1}{\overline{G}_{F}(x+y)}-\frac{1}{\overline{G}_{F}(x)}=\frac{%
G_{F}^{\prime }(z)}{\overline{G}_{F}(x)\overline{G}_{F}(x+y)}y\text{,}
\end{equation*}%
where $x\leq z\leq x+y$. Using $F(z)=G_{F}(z)+\frac{z}{\alpha }G_{F}^{\prime
}(z)$, we obtain that

\begin{eqnarray*}
I &=&\frac{\alpha }{z}\frac{\overline{G}_{F}(z)-\overline{F}(z)}{\overline{G}%
_{F}(x)\overline{G}_{F}(x+y)}y \\
&=&\frac{\alpha }{z}\frac{z^{-\alpha }H_{\alpha }(z)}{\overline{G}_{F}(x)%
\overline{G}_{F}(x+y)}y\textrm{.}
\end{eqnarray*}

Since $m(\alpha )<\infty $, we have $x^{\alpha }\overline{G}%
_{F}(x)\rightarrow m(\alpha )$ and $(x+y)^{\alpha }\overline{G}%
_{F}(x+y)\rightarrow m(\alpha )$ as $x\to\infty$, and $H_{\alpha }(\infty )=m(\alpha )$.
Since $x\sim z\sim x+y$, we obtain that

\begin{equation*}
I=\frac{\alpha }{z}\frac{z^{-\alpha }H_{\alpha }(z)}{\overline{G}_{F}(x)%
\overline{G}_{F}(x+y)}y\sim \frac{\alpha x^{\alpha -1}}{m(\alpha )}y\text{.}
\end{equation*}

Now we consider the second term. We have%
\begin{eqnarray*}
II &=&\frac{\overline{F}(x+y)}{\overline{G}_{F}^{2}(x+y)}-\frac{\overline{F}%
(x)}{\overline{G}_{F}^{2}(x)} \\
&=&\frac{\overline{F}(x+y)-\overline{F}(x)}{\overline{G}_{F}^{2}(x+y)}+%
\overline{F}(x)(\frac{1}{\overline{G}_{F}^{2}(x+y)}-\frac{1}{\overline{G}%
_{F}^{2}(x)}) \\
&=&II_A+\overline{F}(x)II_B\textrm{.}
\end{eqnarray*}

First consider $II_B$. We have
\begin{eqnarray*}
II_B &=&\left(\frac{1}{\overline{G}(x+y)}-\frac{1}{\overline{G}_{F}(x)}\right)\times\left(\frac{1%
}{\overline{G}_{F}(x+y)}+\frac{1}{\overline{G}_{F}(x)}\right) \\
&=&I\times \left(\frac{1}{\overline{G}_{F}(x+y)}+\frac{1}{\overline{G}_{F}(x)}\right)\textrm{.}
\end{eqnarray*}

Using the results of above, we find that%
\begin{equation*}
II_B\sim \frac{2x^{\alpha }}{m(\alpha )}\frac{\alpha x^{\alpha -1}}{%
m(\alpha )}y\textrm{,}
\end{equation*}%
and hence that%
\begin{equation*}
\overline{F}(x)II_B\sim 2x^{\alpha }\overline{F}(x)\frac{\alpha x^{\alpha
-1}}{m^{2}(\alpha )}y=o(1)x^{\alpha -1}\text{.}
\end{equation*}

Now consider $II_A$. We have%
\begin{eqnarray*}
II_A &=&\frac{\overline{F}(x+y)-\overline{F}(x)}{\overline{G}_{F}^{2}(x+y)}
\\
&\sim &\frac{x^{2\alpha }(\overline{F}(x+y)-\overline{F}(x))}{m^{2}(\alpha )}
\\
&=&x^{\alpha -1}\frac{1}{m^{2}(\alpha )}x^{\alpha +1}(\overline{F}(x+y)-%
\overline{F}(x)))\textrm{.}
\end{eqnarray*}

If $x^{\alpha +1}(\overline{F}(x+y)-\overline{F}(x)\rightarrow 0$ as $x\to\infty$, we obtain
that $II_A=o(1)x^{\alpha -1}$. We conclude.

\begin{theorem}\label{teo6}
Let $R(x)$ be the Kendall renewal function defined by (\ref{Kendallrenewalfunction}).
%(6)
Suppose that $m(\alpha )<\infty $ and that $x^{\alpha +1}(\overline{F}%
(x+y)-\overline{F}(x))\rightarrow 0$ as $x\rightarrow \infty $. Then%
\begin{equation*}
\lim_{x\to\infty}\frac{1}{x^{\alpha -1}}(R(x+y)-R(x))= \frac{2\alpha }{m(\alpha )}y%
\text{.}
\end{equation*}%
In particular, if $m(1+\alpha )<\infty $, the result holds.
\end{theorem}

\bigskip

Again we consider two important cases.

\subsection{Regular variation of the density}

Assume that $F$ has a density $f(x)$ and that $xf(x)\big/\overline{F}%
(x)\rightarrow \beta >\alpha $ as $x\to\infty$. In this case, we have $\overline{F}(x)\in
RV_{-\beta }$, and

\begin{equation*}
\overline{F}(x+y)-\overline{F}(x)=f^{\prime }(z)y\sim \beta \frac{\overline{F%
}(z)}{z}y\sim \beta \frac{\overline{F}(x)}{x}y\text{.}
\end{equation*}%
It follows that $x^{\alpha +1}(\overline{F}(x+y)-\overline{F}(x))\sim \beta
yx^{\alpha }\overline{F}(x)\rightarrow 0$ as $x\to\infty$, and Theorem \ref{teo6} applies.

\subsection{The Gamma class}

For $\overline{F}\in \Gamma (g)$ we have $H_{\sigma }(\infty )=m(\sigma
)<\infty $ for all $\sigma >0$ and Theorem \ref{teo6} applies.

%\bigskip

%\pagebreak

\section{Rates of convergence in the Blackwell result}
\label{ratesofconvergenceBlackwell}

To obtain rate of convergence results, we assume that $F(x)$ has a density $%
f(x)$.

\subsection{The derivative of $R(x)$}

\begin{lemma}\label{lemma7}
%(7)
Assume that $F(x)$ has a density $%
f(x)$.
Let $R(x)$ be the Kendall renewal function defined by (\ref{Kendallrenewalfunction}).
We have%
\begin{equation*}
R^{\prime }(x)=\frac{2\alpha x^{\alpha -1}}{m(\alpha )}(1+o(1))+\frac{1}{%
m^{2}(\alpha )}x^{2\alpha }f(x)(1+o(1))\text{.}
\end{equation*}
\end{lemma}

\begin{proof}
Using (\ref{1}), we find%
\begin{equation*}
R^{\prime }(x)=2\frac{G_{F}^{\prime }(x)}{\overline{G}_{F}^{2}(x)}+\frac{f(x)%
}{\overline{G}_{F}^{2}(x)}-2\frac{\overline{F}(x)G_{F}^{\prime }(x)}{%
\overline{G}_{F}^{3}(x)}\text{.}
\end{equation*}%
Using $G_{F}^{\prime }(x)=\alpha x^{-1}(F(x)-G_{F}(x))=\alpha x^{-\alpha
-1}H_{\alpha }(x)$, it follows that%
\begin{eqnarray*}
R^{\prime }(x) &=&2\frac{\alpha x^{-\alpha -1}H_{\alpha }(x)}{\overline{G}%
_{F}^{2}(x)}+\frac{f(x)}{\overline{G}_{F}^{2}(x)}-2\frac{\alpha \overline{F}%
(x)x^{-\alpha -1}H_{\alpha }(x)}{\overline{G}_{F}^{3}(x)} \\
&=&\frac{2\alpha x^{\alpha -1}H_{\alpha }(x)}{x^{2\alpha }\overline{G}%
_{F}^{2}(x)}+\frac{x^{2\alpha }f(x)}{x^{2\alpha }\overline{G}_{F}^{2}(x)}-2%
\frac{\alpha x^{\alpha }\overline{F}(x)x^{\alpha -1}H_{\alpha }(x)}{%
x^{3\alpha }\overline{G}_{F}^{3}(x)} \\
&=&\frac{2\alpha x^{\alpha -1}}{m(\alpha )}(1+o(1))+\frac{1}{m^{2}(\alpha )}%
(1+o(1))x^{2\alpha }f(x)\text{.}
\end{eqnarray*}

This proves the result.
\end{proof}

Specializing to the two cases, we have the following result.

\begin{proposition}\label{prop8}
Assume that $F(x)$ has a density $%
f(x)$.
Let $R(x)$ be the Kendall renewal function defined by (\ref{Kendallrenewalfunction}).
%(8)
\begin{itemize}
	\item[ (i)] If $xf(x)\big/\overline{F}(x)\rightarrow \beta >\alpha $ as $x\to\infty$, then $R^{\prime
}(x)\sim 2\alpha x^{\alpha -1}\big/m(\alpha )$.

	\item[ (ii)] If $f\in \Gamma (g)$ and $\lim \inf_{x\to\infty} x^{\sigma -\alpha -1}g(x)>0$, then $%
R^{\prime }(x)\sim 2\alpha x^{\alpha -1}\big/m(\alpha )$.

\end{itemize}
\end{proposition}

\begin{proof}
(i) In the first case we have $x^{2\alpha }f(x)=O(1)x^{2\alpha -1}\overline{F}%
(x)=o(1)x^{\alpha -1}$ and hence%
\begin{equation*}
R^{\prime }(x)=\frac{2\alpha x^{\alpha -1}}{m(\alpha )}(1+o(1))+o(1)x^{%
\alpha -1}\sim \frac{2\alpha x^{\alpha -1}}{m(\alpha )}\text{.}
\end{equation*}

(ii) If $f\in \Gamma (g)$ we have $\overline{F}(x)\sim f(x)g(x)$ and $%
x^{2\alpha }f(x)=O(1)x^{2\alpha -1}\overline{F}(x)\big/g(x)$. Note that in this
case all moments $m(\sigma )$ are finite.\ It follows that%
\begin{equation*}
x^{1-\alpha }x^{2\alpha }f(x)=O(1)\frac{x^{\sigma }\overline{F}(x)}{%
x^{\sigma -\alpha -1}g(x)}\textrm{.}
\end{equation*}%
If $\lim \inf_{x\to\infty} x^{\sigma -\alpha -1}g(x)>0$, it follows that $x^{1-\alpha
}x^{2\alpha }f(x)\rightarrow 0$ as $x\to\infty$. We conclude that $x^{1-\alpha }R^{\prime
}(x)\rightarrow 2\alpha \big/m(\alpha )$ as $x\to\infty$. This proves the result.
\end{proof}

\bigskip

%Remark.
\begin{remark}
Note that $R^{\prime }(x)\sim 2\alpha x^{\alpha -1}\big/m(\alpha )$ implies that
$R(x+y)-R(x)\sim 2\alpha x^{\alpha -1}y\big/m(\alpha )$.
\end{remark}

\subsection{Rate of convergence}

Now we look for the rate of convergence in the previous Proposition \ref{prop8}.
Using the formula for $R^{\prime }(x)$, earlier we have shown that

\begin{equation*}
R^{\prime }(x)=2\frac{\alpha x^{-\alpha -1}H_{\alpha }(x)}{\overline{G}%
_{F}^{2}(x)}+\frac{f(x)}{\overline{G}_{F}^{2}(x)}-2\frac{\alpha \overline{F}%
(x)x^{-\alpha -1}H_{\alpha }(x)}{\overline{G}_{F}^{3}(x)}\text{.}
\end{equation*}%
It follows that%
\begin{eqnarray*}
R^{\prime }(x)-\frac{2\alpha x^{\alpha -1}}{m(\alpha )} &=&2\frac{\alpha
x^{\alpha -1}H_{\alpha }(x)}{x^{2\alpha }\overline{G}_{F}^{2}(x)}-\frac{%
2\alpha x^{\alpha -1}}{m(\alpha )}+\frac{f(x)}{\overline{G}_{F}^{2}(x)}-2%
\frac{\alpha \overline{F}(x)x^{-\alpha -1}H_{\alpha }(x)}{\overline{G}%
_{F}^{3}(x)} \\
&=&\frac{2\alpha x^{\alpha -1}}{m(\alpha )}\left(\frac{m(\alpha )H_{\alpha }(x)}{%
x^{2\alpha }\overline{G}_{F}^{2}(x)}-1\right)+\frac{f(x)}{\overline{G}_{F}^{2}(x)}%
-2\frac{\alpha \overline{F}(x)x^{-\alpha -1}H_{\alpha }(x)}{\overline{G}%
_{F}^{3}(x)} \\
&=&I+II-III\text{.}
\end{eqnarray*}

First we consider the first term. Recall that earlier we have also shown that%
\begin{equation*}
\frac{1}{\overline{G}_{F}^{2}(x)}-\frac{1}{(x^{-\alpha }m(\alpha ))^{2}}\sim
\frac{2\alpha }{m^{3}(\alpha )}x^{2\alpha }\overline{W}_{\alpha }(x)\text{,}
\end{equation*}%
or equivalently that
\begin{equation*}
\frac{m^{2}(\alpha )}{x^{2\alpha }\overline{G}_{F}^{2}(x)}-1\sim \frac{%
2\alpha }{m(\alpha )}\overline{W}_{\alpha }(x)\textrm{.}
\end{equation*}%
\bigskip

We then have%
\begin{eqnarray*}
I &=&\frac{2\alpha x^{\alpha -1}}{m(\alpha )}\left( \frac{H_{\alpha }(x)}{%
m(\alpha )}\left(\frac{m^{2}(\alpha )}{x^{2\alpha }\overline{G}_{F}^{2}(x)}-1\right)+%
\frac{H_{\alpha }(x)}{m(\alpha )}-1\right)  \\
&=&\frac{2\alpha x^{\alpha -1}}{m(\alpha )}\left( (1+o(1))\frac{2\alpha }{%
m(\alpha )}\overline{W}_{\alpha }(x)-\frac{\overline{H}_{\alpha }(x)}{%
m(\alpha )}\right) \textrm{.}
\end{eqnarray*}

It follows that%
\begin{equation*}
\frac{1}{\overline{W}_{\alpha }(x)}\frac{m(\alpha )}{2\alpha x^{\alpha -1}}%
I=(1+o(1))\frac{2\alpha }{m(\alpha )}-\frac{1}{m(\alpha )}\frac{\overline{H}%
_{\alpha }(x)}{\overline{W}_{\alpha }(x)}\textrm{.}
\end{equation*}

Looking at the last term, note that we have%
\begin{equation*}
\frac{\overline{H}_{\alpha }(x)}{\overline{W}_{\alpha }(x)}=\frac{\alpha
\overline{W}_{\alpha }(x)+x^{\alpha }\overline{F}(x)}{\overline{W}_{\alpha
}(x)}=\alpha +\frac{x^{\alpha }\overline{F}(x)}{\overline{W}_{\alpha }(x)}\textrm{.}
\end{equation*}

Now we distinguish two cases as before.

\subsubsection{Regularly varying case}

If $\overline{F}(x)\in RV_{-\beta },\beta >\alpha $, we have $\overline{W}%
_{\alpha }(x)\sim x^{\alpha }\overline{F}(x)\big/(\beta -\alpha )$ and hence%
\begin{equation*}
\frac{\overline{H}_{\alpha }(x)}{\overline{W}_{\alpha }(x)}=\alpha +\frac{%
x^{\alpha }\overline{F}(x)}{\overline{W}_{\alpha }(x)}\sim \alpha +\beta
-\alpha =\beta \text{.}
\end{equation*}

We conclude that%
\begin{equation*}
\frac{1}{\overline{W}_{\alpha }(x)}\frac{m(\alpha )}{2\alpha x^{\alpha -1}}%
I\rightarrow \frac{2\alpha }{m(\alpha )}-\frac{\beta }{m(\alpha )}=\frac{%
2\alpha -\beta }{m(\alpha )}\text{,}
\end{equation*}%
or equivalently that, as $x\to\infty$,%
\begin{equation*}
\frac{1}{x^{2\alpha -1}\overline{F}(x)}I\rightarrow \frac{2\alpha (2\alpha
-\beta )}{(\beta -\alpha )m^{2}(\alpha )}\text{.}
\end{equation*}

\bigskip

Now we consider the second term:%
\begin{equation*}
II=\frac{x^{2\alpha }f(x)}{x^{2\alpha }\overline{G}_{F}^{2}(x)}\text{.}
\end{equation*}

Using $xf(x)\sim \beta \overline{F}(x)$, we find%
\begin{equation*}
II\sim \frac{x^{2\alpha }f(x)}{m^{2}(\alpha )}\sim \frac{\beta }{%
m^{2}(\alpha )}x^{2\alpha -1}\overline{F}(x)\textrm{.}
\end{equation*}

For the third term, we find%
\begin{eqnarray*}
III &=&2\frac{\alpha x^{\alpha }\overline{F}(x)x^{\alpha -1}H_{\alpha }(x)}{%
x^{3\alpha }\overline{G}_{F}^{3}(x)} \\
&\sim &\frac{2\alpha }{m^{2}(\alpha )}x^{2\alpha -1}\overline{F}(x)\textrm{.}
\end{eqnarray*}

Everything together, we conclude that%
\begin{eqnarray*}
\frac{1}{x^{2\alpha -1}\overline{F}(x)}\left(R^{\prime }(x)-\frac{2x^{\alpha -1}}{%
m(\alpha )}\right) &\rightarrow &\frac{2\alpha (2\alpha -\beta )}{m^{2}(\alpha
)(\beta -\alpha )}+\frac{\beta }{m^{2}(\alpha )}-\frac{2\alpha }{%
m^{2}(\alpha )} \\
&=&\frac{2\alpha (2\alpha -\beta )}{m^{2}(\alpha )(\beta -\alpha )}+\frac{%
\beta -2\alpha }{m^{2}(\alpha )} \\
&=&\frac{(2\alpha -\beta )}{m^{2}(\alpha )}(\frac{2\alpha }{\beta -\alpha }%
-1) \\
&=&\frac{(2\alpha -\beta )(3\alpha -\beta )}{m^{2}(\alpha )(\beta -\alpha )}\textrm{.}
\end{eqnarray*}

As a conclusion we have the following result.

\begin{theorem}\label{teo9}
%(9)
Assume that $F(x)$ has a density $%
f(x)$.
Let $R(x)$ be the Kendall renewal function defined by (\ref{Kendallrenewalfunction}).
If $xf(x)\big/\overline{F}(x)\rightarrow \beta >\alpha $ as $x\to\infty$, then%
\begin{equation*}
\lim_{x\to\infty}\frac{1}{x^{2\alpha -1}\overline{F}(x)}\left(R^{\prime }(x)-\frac{2\alpha
x^{\alpha -1}}{m(\alpha )}\right)= C\textrm{,}
\end{equation*}%
where $C=\frac{(2\alpha -\beta )(3\alpha -\beta )}{m^{2}(\alpha )(\beta
-\alpha )}$.
\end{theorem}

\subsubsection{The Gamma class case $f\in \Gamma (g)$.}

We reconsider the terms $I$, $II$ and $III$ from above.

%i)
For the first term. We have proved that%
\begin{equation*}
\frac{1}{\overline{W}_{\alpha }(x)}\frac{m(\alpha )}{2\alpha x^{\alpha -1}}%
I=(1+o(1))\frac{2\alpha }{m(\alpha )}-\frac{1}{m(\alpha )}\frac{\overline{H}%
_{\alpha }(x)}{\overline{W}_{\alpha }(x)}\textrm{.}
\end{equation*}

Also we have\ $\overline{W}_{\alpha }(x)\sim g(x)x^{\alpha -1}\overline{F}(x)
$ and $\overline{F}(x)\sim f(x)g(x)$.

Clearly we have%
\begin{equation*}
\frac{\overline{H}_{\alpha }(x)}{\overline{W}_{\alpha }(x)}=\alpha +\frac{%
x^{\alpha }\overline{F}(x)}{\overline{W}_{\alpha }(x)}=\alpha +(1+o(1))\frac{%
x}{g(x)}\sim \frac{x}{g(x)}\text{,}
\end{equation*}%
since $g(x)\big/x\rightarrow 0$. It follows that, as $x\to\infty$,%
\begin{equation*}
\frac{g(x)}{x\overline{W}_{\alpha }(x)}\frac{m(\alpha )}{2\alpha x^{\alpha
-1}}I\rightarrow -\frac{1}{m(\alpha )}\text{.}
\end{equation*}

%ii) For $II$
For the second term. We have%
\begin{equation*}
II-\frac{x^{2\alpha }f(x)}{m^{2}(\alpha )}=x^{2\alpha }f(x)\left(\frac{1}{%
x^{2\alpha }\overline{G}_{F}^{2}(x)}-\frac{1}{m^{2}(\alpha )}\right)\textrm{.}
\end{equation*}

Earlier we proved that%
\begin{equation*}
\frac{1}{x^{2\alpha }\overline{G}_{F}^{2}(x)}-\frac{1}{m^{2}(\alpha )}\sim
\frac{2\alpha }{m^{3}(\alpha )}\overline{W}_{\alpha }(x)\text{.}
\end{equation*}

Now we find that
\begin{eqnarray*}
II-\frac{x^{2\alpha }f(x)}{m^{2}(\alpha )} &\sim &\frac{2\alpha x^{2\alpha
}f(x)}{m^{3}(\alpha )}\overline{W}_{\alpha }(x) \\
&\sim &\frac{2\alpha x^{2\alpha }\overline{F}(x)}{g(x)m^{3}(\alpha )}%
\overline{W}_{\alpha }(x)\textrm{,}
\end{eqnarray*}%
and hence, as $x\to\infty$,
\begin{equation*}
\frac{g(x)}{x\overline{W}_{\alpha }(x)}\frac{m(\alpha )}{2\alpha x^{\alpha
-1}}\left(II-\frac{x^{2\alpha }f(x)}{m^{2}(\alpha )}\right)\sim \frac{x^{\alpha }%
\overline{F}(x)}{m^{2}(\alpha )}\rightarrow 0\textrm{.}
\end{equation*}

%iii)
For the third term. We have

\begin{eqnarray*}
III &=&2\frac{\alpha x^{\alpha }\overline{F}(x)x^{\alpha -1}H_{\alpha }(x)}{%
x^{3\alpha }\overline{G}_{F}^{3}(x)} \\
&\sim &2\frac{\alpha x^{\alpha }\overline{F}(x)x^{\alpha -1}m(\alpha )}{%
m^{3}(\alpha )} \\
&\sim &2\frac{\alpha x^{\alpha }\overline{W}_{\alpha }(x)}{g(x)m^{2}(\alpha )%
} \\
&\sim &2\frac{x}{g(x)}\overline{W}_{\alpha }(x)\frac{\alpha x^{\alpha -1}}{%
m^{2}(\alpha )}\textrm{,}
\end{eqnarray*}%
and hence, as $x\to\infty$,%
\begin{equation*}
\frac{g(x)}{x\overline{W}_{\alpha }(x)}\frac{m(\alpha )}{2\alpha x^{\alpha
-1}}III\rightarrow \frac{1}{m(\alpha )}\textrm{.}
\end{equation*}

Everything together we find, as $x\to\infty$,
\begin{equation*}
\frac{g(x)}{x\overline{W}_{\alpha }(x)}\frac{m(\alpha )}{2\alpha x^{\alpha
-1}}\left(R^{\prime }(x)-\frac{2\alpha x^{\alpha -1}}{m(\alpha )}-\frac{%
x^{2\alpha }f(x)}{m^{2}(\alpha )}\right)\rightarrow -\frac{2}{m(\alpha )}\textrm{.}
\end{equation*}%
Equivalently, using $\overline{W}_{\alpha }(x)\sim g(x)x^{\alpha -1}%
\overline{F}(x)$, we have
\begin{equation*}
\frac{1}{x^{2\alpha -1}\overline{F}(x)}\left(R^{\prime }(x)-\frac{2\alpha
x^{\alpha -1}}{m(\alpha )}-\frac{x^{2\alpha }f(x)}{m^{2}(\alpha )}%
\right)\rightarrow -\frac{4\alpha }{m^{2}(\alpha )}\textrm{.}
\end{equation*}

Note that we have, as $x\to\infty$,%
\begin{equation*}
\frac{x^{2\alpha -1}\overline{F}(x)}{x^{2\alpha }f(x)}=\frac{\overline{F}(x)%
}{xf(x)}\sim \frac{g(x)}{x}\rightarrow 0\text{.}
\end{equation*}%
\bigskip

It follows that, as $x\to\infty$,%
\begin{equation*}
\frac{m^{2}(\alpha )}{x^{2\alpha }f(x)}\left(R^{\prime }(x)-\frac{2\alpha
x^{\alpha -1}}{m(\alpha )}\right)\rightarrow 1\text{.}
\end{equation*}

We have proved the following result.

\begin{theorem}\label{teo10}
%(10)
Assume that $F(x)$ has a density $%
f(x)$.
Let $R(x)$ be the Kendall renewal function defined by (\ref{Kendallrenewalfunction}).
Suppose that $f\in \Gamma (g)$. Then
\begin{equation*}
\lim_{x\to\infty}\frac{1}{x^{2\alpha -1}\overline{F}(x)}\left(R^{\prime }(x)-\frac{2\alpha
x^{\alpha -1}}{m(\alpha )}-\frac{x^{2\alpha }f(x)}{m^{2}(\alpha )}%
\right)= -\frac{4\alpha }{m^{2}(\alpha )}\textrm{,}
\end{equation*}%
and%
\begin{equation*}
\lim_{x\to\infty}\frac{m^{2}(\alpha )}{x^{2\alpha }f(x)}\left(R^{\prime }(x)-\frac{2\alpha
x^{\alpha -1}}{m(\alpha )}\right)= 1\text{.}
\end{equation*}
\end{theorem}

%\bigskip

\subsection{Rate of convergence in Blackwell's result}

Earlier, we have proved a Blackwell type of result, i.e. $R(x+y)-R(x)\sim
2\alpha x^{\alpha -1}y\big/m(\alpha )$. We want to use Theorems \ref{teo9} and \ref{teo10} to
find a rate of convergence result here.

\subsubsection{Regularly varying case}

Clearly for $y>0$ we have $R(x+y)-R(x)=\int_{x}^{x+y}R^{\prime }(t)dt$.
Using Theorem \ref{teo9}, we see that%
\begin{equation*}
\int_{x}^{x+y}\left(R^{\prime }(t)-\frac{2\alpha t^{\alpha -1}}{m(\alpha )}%
\right)dt\sim C\int_{x}^{x+y}t^{2\alpha -1}\overline{F}(t)dt\sim Cx^{2\alpha -1}%
\overline{F}(x)y\textrm{.}
\end{equation*}%
It also follows that%
\begin{eqnarray*}
&&R(x+y)-R(x)-\frac{2\alpha }{m(\alpha )}x^{\alpha -1}y \\
&=&\int_{x}^{x+y}\left(R^{\prime }(t)-\frac{2\alpha t^{\alpha -1}}{m(\alpha )}\right)dt+%
\frac{2\alpha }{m(\alpha )}\int_{x}^{x+y}(t^{\alpha -1}-x^{\alpha -1})dt\textrm{.}
\end{eqnarray*}

Hence we find that%
\begin{eqnarray*}
&&R(x+y)-R(x)-\frac{2\alpha }{m(\alpha )}x^{\alpha -1}y \\
&=&(1+o(1))Cx^{2\alpha -1}\overline{F}(x)y+(1+o(1))\frac{2\alpha (\alpha -1)%
}{m(\alpha )}x^{\alpha -2}y^{2}\textrm{,}
\end{eqnarray*}

or%
\begin{eqnarray*}
&&x^{1-\alpha }(R(x+y)-R(x))-\frac{2\alpha }{m(\alpha )}y \\
&=&(1+o(1))Cx^{\alpha }\overline{F}(x)y+(1+o(1))\frac{2\alpha (\alpha -1)}{%
m(\alpha )}x^{-1}y^{2}\textrm{.}
\end{eqnarray*}%

\bigskip

If $x^{1+\alpha }\overline{F}(x)\rightarrow 0$, we find that
\begin{equation*}
x\left( x^{1-\alpha }(R(x+y)-R(x))-\frac{2}{m(\alpha )}y\right) \rightarrow
\frac{2\alpha (\alpha -1)}{m(\alpha )}y^{2}\text{.}
\end{equation*}%
\bigskip

This is the case when $m(1+\alpha )<\infty $.

\bigskip

If $x^{1+\alpha }\overline{F}(x)\rightarrow D$, where $0<D\leq \infty $,
we find that
\begin{eqnarray*}
&&\frac{1}{x^{\alpha }\overline{F}(x)}\left( x^{1-\alpha }(R(x+y)-R(x))-%
\frac{2\alpha }{m(\alpha )}y\right)  \\
&\rightarrow &Cy+\frac{2\alpha (\alpha -1)}{Dm(\alpha )}y^{2}\text{.}
\end{eqnarray*}

\subsubsection{The Gamma class case}

Using Theorem \ref{teo10}, we have%
\begin{eqnarray*}
\int_{x}^{x+y}\left(R^{\prime }(t)-\frac{2\alpha t^{\alpha -1}}{m(\alpha )}\right)dt
&\sim &\frac{1}{m^{2}(\alpha )}\int_{x}^{x+y}t^{2\alpha }f(t) \\
&\sim &\frac{x^{2\alpha }}{m^{2}(\alpha )}(F(x+y)-F(x))\textrm{.}
\end{eqnarray*}

It follows that%
\begin{eqnarray*}
R(x+y)-R(x)-\frac{2\alpha x^{\alpha -1}}{m(\alpha )}y &=&(1+o(1))\frac{%
x^{2\alpha }}{m^{2}(\alpha )}(F(x+y)-F(x)) \\
&&+(1+o(1))\frac{2\alpha (\alpha -1)}{m(\alpha )}x^{\alpha -2}y^{2}\textrm{,}
\end{eqnarray*}%
and, as $x\to\infty$,%
\begin{eqnarray*}
\lefteqn{x^{2-\alpha }\left( R(x+y)-R(x)-\frac{2\alpha x^{\alpha -1}}{m(\alpha )}%
y\right)} \\  &=&(1+o(1))\frac{x^{2+\alpha }}{m^{2}(\alpha )}(F(x+y)-F(x)) %\\
%&&
+(1+o(1))\frac{2\alpha (\alpha -1)}{m(\alpha )}y^{2} \\
&\rightarrow &\frac{2\alpha (\alpha -1)}{m(\alpha )}y^{2}\textrm{,}
\end{eqnarray*}%
since all moments are finite it follows that $x^{2+\alpha
}(F(x+y)-F(x))\rightarrow 0$ as $x\to\infty$.

\section{Concluding remarks}

\begin{enumerate}
	\item %1)
	We can also study weighted renewal functions of the form%
\begin{equation*}
WR(x)=\sum_{n=0}^{\infty }a_{n}F^{\boxtimes n}(x)\textrm{,}
\end{equation*}%
where $(a_{n})$ is a sequence of positive numbers. In this case the $G-$%
transform of $WR(x)$ is given by%
\begin{equation*}
G_{WR}(x)=\sum_{n=0}^{\infty }a_{n}G_{F}^{n}(x)=A(G_{F}(x))\text{,}
\end{equation*}%
where $A(z)=\sum_{n=0}^{\infty }a_{n}z^{n}$ is the generating function of
the sequence of weights.

	\item %2)
Recall that $F^{\boxtimes n}(x)$ is the d.f. of the Kendall sum $S_{\boxtimes
n}$ of independent and identically distributed random elements. If $(a_{n})$ is the probability density function of a discrete random
variable $N$, then $WR(x)$ is the d.f. of the random sum $S_{\boxtimes N}$.

	\item %3)
The Williamson transform can be written as $G_{F}(x)=\int_{0}^{x}P(Z\geq
t/x)dF(t)=P(X/Z\leq x)$, where $Z$ denotes a positive r.v. with d.f. $%
P(Z\leq x)=x^{\alpha },0\leq x\leq 1$. It could be\ interesting to study a
transform where we replace this d.f. of $Z$ by another d.f.

\end{enumerate}

%\pagebreak

%\section{References}

%\begin{enumerate}
%\item Bingham,\ N., Goldie, C., Teugels, J.,Regular Variation.\ Cambridge
%University Press 1989.

%\item Geluk, J.G.\ and\ de Haan, L., Regular variation, Extensions and
%Tauberian theorems.\ CWI\ Tract 40, Centre for Mathematics and Computer
%Science, Amsterdam (1987).

%\item B.H.\ Jasiulis-Go{\l}dyn, J.K.\ Misiewicz, K.\ Naskr{\k e}ta, E.\ Omey.\
%Renewal theory for extremal Markov sequences of Kendall type.\ Stoch.\
%Proc.\ and their Appl. 130, 3277 - 3294, 2020.

%\item Kevei, P. On a conjecture of Seneta.\ arXiv:2007.04668 [math.\ PR]
%(2020).

%\item Meitner Cadena and Edward Omey, A Seneta's conjecture and the
%Williamson transform. Submitted (August 2020).

%\item R.\ Williamson.\ Multiple monotone functions and their Laplace
%transforms. Duke Math.\ J. 23, 189-207 (1956).
%\end{enumerate}

\section*{Disclosure statement}
No potential conflict of interests was reported by the authors.


\begin{thebibliography}{10}
\bibitem{ArenJasOmey}
M. Arendarczyk, B. Jasiulis-Go{\l}dyn, E. Omey, \emph{ Asymptotic properties of extremal Markov processes driven by Kendall convolution},Submitted (October 2020).

\bibitem{Asmussen} Asmussen, S. \emph{Ruin Probabilities}. World Scientific (2000).

\bibitem{Bingham} Bingham, N., Goldie, C., Teugels, J. \emph{Regular Variation}. Cambridge
University Press (1989).

\bibitem{deHaan} de Haan., L. \emph{On regular variation and its applications to the weak convergence of sample extremes}.
Mathematical Centre Tracts, 32 (1970).

\bibitem{Geluk} Geluk, J.G., de Haan, L. \emph{Regular variation, Extensions and
Tauberian theorems}. CWI Tract 40, Centre for Mathematics and Computer
Science, Amsterdam (1987).

\bibitem{Grubel}
Gr\"{u}bel, R. \emph{Functions of discrete probability measures: Rates of convergence in the renewal theorem}.
Zeitschrift F\"{u}r Wahrscheinlichkeitstheorie Und Verwandte Gebiete 64, 341-357 (1983).

\bibitem{Jasiulis3}
Jasiulis-Go{\l}dyn, B. H., Misiewicz, J. K.
\emph{On the uniqueness of the Kendall generalized
convolution}, J. Theoret. Probab. 24, 746-755 (2011).

\bibitem{Jasiulis2}
Jasiulis-Go{\l}dyn, B. H., Misiewicz, J. K.
\emph{Kendall random walk, Williamson transform, and the corresponding Wiener-Hopf factorization}.
Lithuanian Mathematical Journal 57, 479-489 (2017).

\bibitem{Jasiulis}
Jasiulis-Go{\l}dyn, B. H., Misiewicz, J. K., Naskr{\k e}t, K., Omey, E.
\emph{Renewal theory for extremal Markov sequences of Kendall type}.
Stoch. Proc. and their Appl. 130, 3277-3294 (2020).

\bibitem{Kalashnikov}
Kalashnikov, V. V.
\emph{Uniform Estimation of the Convergence Rate in a Renewal Theorem for the Case of Discrete Time}.
Theory of Probability \& Its Applications 22, 390-394 (1978).

\bibitem{Karamata}
Karamata, J.
\emph{Sur un mode de croissance r\'{e}guli\'{e}re des fonctions}.
Mathematica (Cluj) 4, 38-53 (1930).

\bibitem{Kendall}
Kendall, D.G.
\emph{Foundations of a theory of random sets}.
In: Harding, E.F., Kendall, D.G. (eds.) Stochastic
Geometry, 322-376. Willey, New York (1974).

\bibitem{Kevei}
Kevei, P.
\emph{On a conjecture of Seneta}.
arXiv:2007.04668 [math.\ PR]
(2020).

\bibitem{Kucharczak}
Kucharczak, J., Urbanik, K.
\emph{Quasi-stable functions}. Bull. Pol. Acad. Sci., Math. 22, 263-268
(1974).

\bibitem{Mitov}
Mitov, K. V., Omey, E. A. M. \emph{Renewal Processes}, Springer International Publishing (2014).

\bibitem{Omey}
Omey, E., Cadena, M. \emph{A Seneta's conjecture and the
Williamson transform}. Submitted (August 2020).

\bibitem{Rogozin1973}
Rogozin, B. A.
\emph{An Estimate of the Remainder Term in Limit Theorems of Renewal Theory}.
Theory of Probability \& Its Applications 18, 662-677 (1974).

\bibitem{Rogozin1977}
Rogozin, B. A.
\emph{Asymptotics of Renewal Functions}. Theory of Probability \& Its Applications 21, 669-686 (1977).

\bibitem{Rolski}
Rolski, T., Schmidt, V., Schmidli, H., Teugels, J. \emph{Stochastic Processes for Insurance and Finance}, J.Wiley \&
Sons, Chichester (1999).

\bibitem{Urbanik}
Urbanik, K. \emph{Generalized convolutions}, Studia Math. 23, 217-245 (1964).

\bibitem{Williamson}
Williamson, R. \emph{Multiple monotone functions and their Laplace
transforms}. Duke Math.\ J. 23, 189-207 (1956).

\end{thebibliography}
\end{document}